\documentclass{amsart}
\usepackage[utf8]{inputenc}
\usepackage{amsmath}
\usepackage{amsfonts}
\usepackage{bbm}
\usepackage{amsthm}
\usepackage{graphicx}
\usepackage{tikz}
\usepackage{caption}
\usepackage{subfig}
\usepackage[autostyle]{csquotes}

\usepackage[
text={440pt,575pt},
headheight=9pt,
centering
]{geometry}
\usepackage[autostyle]{csquotes}

\usepackage{xcolor}
\usepackage{hyperref}
\usepackage{url, hypcap}
\definecolor{darkblue}{RGB}{0,0,160}
\hypersetup{
colorlinks,%
citecolor=darkblue,%
filecolor=black,%
linkcolor=darkblue,%
urlcolor=darkblue
}
\newcommand{\excise}[1]{}

\usepackage{mathtools}
\newcommand{\on}[1]{\operatorname{#1}}

\usepackage[capitalize]{cleveref}
\newcommand{\cox}[1]{\mathtt{#1}}

\newtheorem{thm}{Theorem}

\theoremstyle{definition}

\newtheorem{remark}[thm]{Remark}
\newtheorem{defn}[thm]{Definition}

\numberwithin{equation}{section}

\begin{document}
\title{Asymptotics of a Locally Dependent Statistic on Finite Reflection Groups}

\author[F.~R\"ottger]{Frank R\"ottger}
\address[]{Fakultät für Mathematik\\Otto-von-Guericke Universität
Magdeburg\\39106 Magdeburg\\Germany}
\email{frank.roettger@ovgu.de}
\urladdr{\url{http://www.imst3.ovgu.de/Arbeitsgruppe/Frank+Rottger}}

\begin{abstract}
This paper discusses the asymptotic behaviour of the number of descents in a random signed permutation and its inverse, which was posed as an open problem by Chatterjee and Diaconis in \cite{chatterjee2017}. For that purpose, we generalize their result for the asymptotic normality of the number of descents in a random permutation and its inverse to other finite reflection groups. This is achieved by applying their proof scheme to signed permutations, i.e.~elements of Coxeter groups of type $ \cox{B}_n $, which are also known as the hyperoctahedral groups.  Furthermore, a similar central limit theorem for elements of Coxeter groups of type $\cox{D}_n$ is derived via Slutsky's Theorem and a bound on the Wasserstein distance of certain normalized statistics with local dependency structures and bounded local components is proven for both types of Coxeter groups.
In addition, we show a two-dimensional central limit theorem via the Cram\'er-Wold device.
\end{abstract}

\maketitle
\setcounter{tocdepth}{1}

\section{Introduction}

A recent result of Chatterjee and Diaconis in \cite{chatterjee2017} was a new proof of the asymptotic normality of the number of descents in a random permutation and its inverse, normalized by its expected value and its variance. This was shown via the \emph{method of interaction graphs} and a bound on the Wasserstein distance to the standard normal distribution firstly introduced in \cite{chatterjee2008}. The \emph{method of interaction graphs} and the bound on the Wasserstein distance are shortly summarized in Section \ref{s:interactiongraphs}. The asymptotic normality of the number of descents in a random permutation and its inverse was already shown by Vatutin via generating functions in \cite{Vatutin} in 1996, but was not generalized to other statistics depending on both a random permutation and its inverse. Chatterjee and Diaconis showed such a generalization through a bound of the Wasserstein distance to the standard normal distribution for a wider class of normalized statistics that depend on a random permutation and its inverse. In the last section of \cite{chatterjee2017}, they issued the asymptotic normality of the number of descents in an element of a finite reflection group and its inverse, for example for random signed permutations, as an open problem and indicated, that their approach should also suffice in this case. This paper confirms their intution by applying their proof scheme on signed permutations, that is elements of the Coxeter group of type $ \cox{B}_n $. For that purpose we construct random signed permutations and their inverses from the same random variables. With this construction, we are able to apply the \emph{method of interaction graphs}, exactly like Chatterjee and Diaconis. Together with the bound on the Wasserstein distance between the normalized statistic and a standard normal distribution mentioned before (see Theorem \ref{t:chatterjee}), we can show the asymptotic normality by plugging in the formulas for the variance of the sum of the statistics into the bounds. Kahle and Stump listed the expected values and variances of the sum of the statistics for all finite irreducible Coxeter groups in \cite[Corollary 5.2]{KahleStump2018}.\\
Using this result for signed permutations, we can extend the result to elements of the Coxeter group of type $ \cox{D}_n $, which are signed permutations with an even number of negative signs. This is done via an application of Slutsky's Theorem (see Theorem \ref{t:CLT_D}).\\
To generalize these results to certain sums of statistics of both a random signed permutation and its inverse, which have a bounded local degree and local components which are bounded by $ 1 $, we again follow Chatterjee and Diaconis and modify the interaction graphs in the right way so that we can apply Theorem \ref{t:chatterjee}. From this, we show that this also works for elements of Coxeter groups of type $ \cox{D}_n $ (see Section \ref{s:generalization}). The last section discusses the asymtotic behaviour of the two-dimensional statistic formed by the number of descents in an element of a Coxeter group of both type $ \cox{B}_n $ and $ \cox{D}_n $ and its inverse via the Cram\'er-Wold device. 

\subsection*{Acknowledgements} 
I want to thank Philipp Godland, Hauke Seidel and in particular Norbert Gaffke for helpful comments and discussions. Furthermore, I want to thank my PhD-advisors Thomas Kahle and Rainer Schwabe for their support and guidance. \\
As a fellow of the research training group on Mathematical Complexity Reduction at the Otto-von-Guericke-University Magdeburg, I am funded by the Deutsche Forschungsgemeinschaft (DFG, German Research Foundation) - 314838170, GRK 2297 MathCoRe.

\section{Interaction Graphs}\label{s:interactiongraphs}

We give a short overview over the method of interaction graphs as it is presented in \cite{chatterjee2017}. Let $ (\mathcal{X},\mathcal{A}) $ be a measurable space and $ f:\mathcal{X}^{n}\rightarrow \mathbb{R} $ a measurable map. Consider a map $ G(x) $, which connects every $ x \in \mathcal{X}^n $ with a simple graph on $ [n]:=\{1,2,\ldots,n\} $.
This \emph{graphical rule} is \emph{symmetric}, if for a permutation $ \pi $ the graph $ G(x_{\pi(1)},\ldots,x_{\pi(n)}) $ has the edge set $ \{(\pi(i),\pi(j))|~ (i,j) ~\text{is an edge of }~G(x_{1},\ldots,x_{n})\} $. For $ m\ge n $, let  $ G'(x) $ for $ x \in \mathcal{X}^m $ be a symmetric graphical rule on $ \mathcal{X}^m $. $ G'(x) $ is an \emph{extension} of $ G(x) $, if $G(x)= G(x_1,\ldots,x_n) $ is a subgraph of $ G'(x)= G'(x_1,\ldots,x_m) $ for all $ x \in \mathcal{X}^m $. 
To define an interaction rule, let for $ x,x' \in \mathcal{X}^n $ 
\[  x^i:=(x_1,\ldots,x_{i-1},x'_{i},x_{i+1},\ldots,x_{n}). \]
Furthermore, let $ x^{ij} $ be the vector $ x $ with replacements in the $ i $-th and $ j $-th position. Then, $ i $ and $ j $ are \emph{non-interacting}, if
\[f(x)-f(x^j)=f(x^i)-f(x^{ij}).\]
A graphical rule $ G $ is an \emph{interaction rule} for a function $ f $, if for any $ x,x' \in \mathcal{X}^n $ and any $ i,j $, the edge $ (i,j) $ not being an edge of either $ G(x),G(x^i),G(x^j) $ or $ G(x^{ij}) $ implies that $ i $ and $ j $ are non-interacting.

The Wasserstein distance is a distance function on the space of probability measures \cite[Chapter~7]{ambrosiogiglisavare}.

\begin{defn}[Wasserstein distance, also known as Kantorovich–Rubinstein metric] 
	Let $ (M,d) $ be a metric space where every probability measure is a Radon measure and let $ P_p(M) $ be the collection of probability measures on $ M $ with finite $ p $-th moments.
	The $ L^p $-Wasserstein distance between $ X \sim \mu\in P_p(M)  $ and $ Y\sim \nu \in P_p(M) $ is defined as
	\[\delta_p(\mu,\nu)=\left(\inf \mathbb{E} \left[d(X,Y)^p\right] \right)^{\frac{1}{p}}, \]
	where the infimum is taken over all joint distributions of $ (X,Y)^T $ on $ M\times M $ with marginals $ \mu $ and~$ \nu $.
\end{defn}

\begin{defn}\label{d:rank-statistic}
	Let $ Y=(Y_1,\ldots,Y_n) $ be a vector of real-valued random variables distributed according to a continuous distribution. The rank statistic is defined as $ R(Y_i)=\sum_{j=1}^{n}\mathbbm{1}_{\{Y_i\,\ge \,Y_j\}} $, where $\mathbbm{1}_{\lbrace \cdot \rbrace}$ denotes the indicator function. The value of $ R(Y_i) $ gives the position of $ Y_i $ when $ Y $ is sorted in ascending order.
\end{defn}

We later apply the following theorem from \cite{chatterjee2008}, which can also be found in \cite{chatterjee2017}, on signed permutations. The Theorem gives a bound on the Wasserstein distance between a normalized statistic that admits a graphical interaction rule and the standard normal distribution.
\begin{thm}[Chatterjee]\label{t:chatterjee}
Let $ f:\mathcal{X}^n \rightarrow \mathbb{R}$ be a measurable map that admits a symmetric interaction rule $ G(x) $. Let $ X_1,X_2,\ldots $ be independent and identically distributed $ \mathcal{X} $-valued random variables and let $ X:=(X_1,\ldots,X_n) $. Let $ W:=f(X) $ and $ \sigma^{2}:=\mathbb{V}(W) $. Let $ X'=(X_1',\ldots,X_n') $ be an independent copy of $ X $.
For each $ j $, define 
\[\Delta_{j}f(X)=W-f(X_1,\ldots,X_{j-1},X_j',X_{j+1},\ldots,X_n)\]
and let $ M:=\max_{j}|\Delta_j f(X)| $. Let $ G'(x) $ be an extension of $ G(x) $ on $ \mathcal{X}^{n+4} $ and define
\[\delta :=1+~\text{degree of the vertex $ 1 $ in $ G'(X_1,\ldots,X_{n+4})$}.\]
Then, the Wasserstein distance $ \delta_W $ between $ \frac{W-\mathbb{E}(W)}{\sigma} $ and N$ (0,1) $ satisfies
\[\delta_W \le \frac{C\sqrt{n}}{\sigma^2}\mathbb{E}(M^8)^{\frac{1}{4}}\mathbb{E}(\delta^4)^{\frac{1}{4}}+\frac{1}{2\sigma^3}\sum_{j=1}^{n} \mathbb{E}|\Delta_j f(X)|^3 \]
for some constant $ C $ independent of $ n $.
\end{thm}

Chatterjee and Diaconis used the theorem above to show a central limit theorem for statistics of the form $ F_1(\pi)+F_2(\pi^{-1}) $, where both $ F_1 $ and $ F_2 $ have bounded local degree and their local components' absolute values are bounded by $ 1 $. Hereby $ \pi $ denoted a permutation, hence an element of a Coxeter group of type $ \cox{A}_n $. We apply the same proof scheme to statistics on signed permutation by modifying their model.

\section{Signed Permutations}

Chatterjee and Diaconis modeled elements of the symmetric group $ S_n = \cox{A}_{n-1} $ and their inverses by ranking functions on series of uniformly distributed random variables on the unit square. We slightly modify this model by additionally introducing a random sign.
The Coxeter group of type $\cox{B}_n$ is the symmetry group of the $ n $-hypercube. 
It is isomorphic to the signed permutation group of rank $ n $, which is the subgroup of all permutations on $\{ \pm 1,\ldots,\pm n\} $ with the antisymmetric constraint $ -\tilde{\pi}(i)=\tilde{\pi}(-i) $. In a one-line notation we write $ \tilde{\pi}=(\tilde{\pi}(1),\ldots,\tilde{\pi}(n)) $ where $ \tilde{\pi}(i)\in \{ \pm 1,\ldots,\pm n\} $ and $ \{|\tilde{\pi}(1)|,\ldots,|\tilde{\pi}(n)|\}=[n] $.
Following \cite[Proposition 8.1.2]{BB:CombinatoricsCoxetergroups}, it holds that the descents in some signed permutation $ \tilde{\pi} \in \cox{B}_n $ in the one-line notation are
\[\on{Des}(\tilde{\pi})=\{0\le i < n: \tilde{\pi}(i)> \tilde{\pi}(i+1) \},\] 
where $ \tilde{\pi}(0)=0 $. We write for $ \tilde{\pi} \in \cox{B}_n $
\begin{align}
\on{des}(\tilde{\pi})=|\on{Des}(\tilde{\pi})|=\mathbbm{1}_{\{0>\tilde{\pi}(1)\}}+\sum_{i=1}^{n-1}\mathbbm{1}_{\{\tilde{\pi}(i)>\tilde{\pi}(i+1)\}}.
\end{align}
In the following theorem, we study the asymptotic behaviour of the statistic
\[t(\tilde{\pi})=\on{des}(\tilde{\pi})+\on{des}(\tilde{\pi}^{-1}).\]
If $ \tilde{\pi} $ is picked uniformly from $ \cox{B_n} $, the statistic $ t(\tilde{\pi}) $ gives rise to a random variable $ T_{\cox{B}_n} $. 
We show a central limit theorem for the sequence $ (T_{\cox{B}_n})_n $, normalized by its expected value and its variance, so
\begin{align}\label{conv:CLT}
\frac{T_{\cox{B}_n}-\mathbb{E}(T_{\cox{B}_n})}{\sqrt{\mathbb{V}(T_{\cox{B}_n})}}\stackrel{D}{\rightarrow} N(0,1),
\end{align}
by adapting the proof of Theorem 1.1 in \cite{chatterjee2017} for the modified model. 

\begin{thm}\label{t:CLT_B}
	Given a sequence of Coxeter groups of type $  \cox{B}_n $ of growing rank. Then,
	$ T_{\cox{B}_n} $
	satisfies the central limit theorem, if $ n $ tends to infinity. 
\end{thm}

\begin{proof}

Let $ \mathcal{X}:=[0,1]^2 \times \{-1,1\} $ and $ X_1,X_2,\ldots $ be independent and identically distributed of the form $ (U_i,V_i,B_i) $ with $ (U_i,V_i)\sim\,$Unif$\left( [0,1]^2\right)$ and $ B_i\sim\,$Ber$ (\frac{1}{2}) $ on $ \{-1,1\} $ and independent of $ (U_i,V_i) $. Let $ X:=(X_1,\ldots,X_n) $ and let the x-rank of $ X_i $ be the rank statistic (cf.~\cref{d:rank-statistic}) of $ U_i $ among $ (U_1,\ldots,U_n) $ and the y-rank of $ X_i $ the rank statistic of $ V_i $ among $ (V_1,\ldots,V_n) $, so that 
$ X_{(1)},\ldots,X_{(n)} $ denote the $ X_i $ ordered with respect to their x-ranks and $ X^{(1)},\ldots,X^{(n)} $ with respect to their y-ranks. This means that $ \pi(i)=\text{y-rank of}~ X_{(i)} $ is a random permutation and $ \sigma(i)=\text{x-rank of}~ X^{(i)} $ is its inverse.
Now, to see that
\begin{equation*}
\tilde{\pi}(i):=B_{(|i|)}\text{sign}(i)\pi(|i|), \qquad \tilde{\sigma}(i):=B^{(|i|)}\text{sign}(i)\sigma(|i|)
\end{equation*}
define random signed permutations, just check that $ \tilde{\pi}(-i)=-\tilde{\pi}(i) $ and $ \tilde{\sigma}(-i)=-\tilde{\sigma}(i) $ and that $ \tilde{\pi}(i) $ and $ \tilde{\sigma}(i) $ are injective.
Furthermore it follows that $ \tilde{\sigma}=\tilde{\pi}^{-1} $, as $ B_{(\sigma(|i|))}=B^{(|i|)} $ and
\[\tilde{\pi}(\tilde{\sigma}(i))=B_{(|\tilde{\sigma}(i)|)}\text{sign}(\tilde{\sigma}(i))\pi(|\tilde{\sigma}(i)|)=B_{(\sigma(|i|))}\text{sign}(B^{(|i|)}\text{sign}(i))\pi(\sigma(|i|))=i.\]

Therefore the number of descents in the signed permutation and its inverse is given by:
\begin{align}\label{def:W_B}
T_{\cox{B}_n}&:=f(X)=\sum_{i=0}^{n-1}\mathbbm{1}_{\{\tilde{\pi}(i)> \tilde{\pi}(i+1)\}}+\sum_{i=0}^{n-1}\mathbbm{1}_{\{\tilde{\sigma}(i)> \tilde{\sigma}(i+1)\}}\\
&=\mathbbm{1}_{\{0>B_{(1)}\pi(1)\}}+\sum_{i=1}^{n-1}\mathbbm{1}_{\{B_{(i)}\pi(i)>B_{(i+1)}\pi(i+1)\}}+\mathbbm{1}_{\{0>B^{(1)}\sigma(1)\}}+\sum_{i=1}^{n-1}\mathbbm{1}_{\{B^{(i)}\sigma(i)>B^{(i+1)}\sigma(i+1)\}} \nonumber
\end{align}

For $ x\in \mathcal{X}^n $, define a simple graph $ G(x) $ on $ [n] $ as follows: For any $ 1\le i \neq j \le n $, let $ \{i,j\} $ be an edge if and only if the x-rank of $ x_i $ and the x-rank of $ x_j $ or the y-rank of $ x_i $ and the y-rank of $ x_j $ differ by at most $ 1 $. To check that this graphical rule is symmetric, see that the edge set of a relabeled Graph $ G(x_{\pi(1)},\ldots,x_{\pi(n)}) $, where $ \pi $ is an arbitrary permutation, has the edge set $ \{(\pi(i),\pi(j))|~ (i,j) ~\text{is an edge of }~G(x_{1},\ldots,x_{n})\} $. This is true, since the x-ranks or the y-ranks of $ x_{\pi(i)} $ are equal to the respective ranks of $ x_i $. Hence this graph is invariant under relabeling of the indices and it is therefore a symmetric graphical rule.
Given $ x,x'\in \mathcal{X}^n $, $ x^{i} $ is the vector $ (x_1,\ldots x_{i-1},x'_i,x_{i+1},\ldots,x_n) $, so the vector $ x $ in which the $ i $-th entry is replaced by the $ i $-th entry of $ x' $. Furthermore, $ x^{ij} $ is the vector with replacements in the $ i $-th and the $ j $-th entry.
Now, suppose that $ (i,j) $ is not an edge in $ G(x),G(x^{i}),G(x^{j}) $ or $ G(x^{ij}) $. Then, the equation
\[f(x)-f(x^j)=f(x^i)-f(x^{ij})\]
holds, as $ j $ is not a neighbour of $ i $ in either of the four graphs. To better visualize this, check that
\begin{align}\label{eq:symmetric}
f(x)=f(x^i)+f(x^j)-f(x^{ij}).
\end{align}
Any indicator function in $ f(x) $, that is not dependent of either $ x_i $ or $ x_j $, appears in $ f(x^i),f(x^j) $ and $ f(x^{ij}) $, as it is left unchanged by the replacements in $ x^i,x^j $ or $ x^{ij} $. Those indicator functions, that depend on $ x_i $ but not on $ x_j $, are unchanged in $ f(x^j) $. As $ i $ and $ j $ are no neighbours in all four graphs, these indicator functions, that depend on $ x^i $ but not on $ x^j $, appear in both $ f(x^i) $ and $ f(x^{ij}) $. Therefore, the indicator functions that either depend on $ x_i $ or on $ x_j $ turn up exactly once on both sides of the equation. Hence Equation (\ref{eq:symmetric}) holds, since there cannot be an indicator functions that depend on both $ x^i $ and $ x^j $, as $ i $ and $ j $ are no neighbours in all four graphs.
This means, that $ G(x) $ is a symmetric interaction rule for $ f $. 
Now, we construct an extension $ G'(x) $ of $ G(x) $ on $ \mathcal{X}^{n+4} $. For any $ 1\le i \neq j \le n+4 $, let $ \{i,j\} $ be an edge in $ G'(x) $ if and only if the x-rank of $ x_i $ and the x-rank of $ x_j $ or the y-rank of $ x_i $ and the y-rank of $ x_j $, differ by at most $ 5 $. As this graph is invariant under relabeling of the indices, it is a symmetric graphical rule. Obviously, every edge in $ G(x) $ is also an edge in $ G'(x) $, as the distance between two connected nodes in $ G(x) $ can be $ 5 $ at most through the insertion of four additional nodes. Therefore $ G'(x) $ is an extension of $ G(x) $.
As $ T_{\cox{B}_n} $ and $ f(X_1,\ldots,X_{j-1},X_j',X_{j+1},\ldots,X_n) $ can differ in at most $ 4 $ summands, $ |\Delta_{j}f(X)|\le 4 $. Furthermore, the degree of any node in $ G'(x) $ is bounded by $ 20 $, as either the difference in the x-ranks or in the y-ranks has to be smaller or equal to $ 5 $. This means, that $ |\delta|\le 21 $. Then, by Theorem \ref{t:chatterjee},
\[\delta_{T_{\cox{B}_n}} \le \frac{C \sqrt{n}}{\sigma^2}+\frac{Cn}{\sigma^3}\]
for some constant $ C $. As \cite{KahleStump2018} shows, $ \sigma^2=\mathbb{V}(T_{\cox{B}_n})=\frac{n+3}{6} $. Therefore, $ T_{\cox{B}_n} $ follows the central limit theorem.
 
\end{proof}

\section{Coxeter Group of Type \texorpdfstring{$\cox{D}_n$}{Dn}}

This section reproduces the previous section's result for sequences of Coxeter groups of type $ \cox{D}_n $. The Coxeter group of type $\cox{D}_n$ is the symmetry group of the $ n $-demicube. It is isomorphic to the subgroup of the signed permutation group of rank $ n $ that consist of all signed permutation with an even number of negative signs. This means, that
\[\cox{D}_n=\{\pi\in \cox{B}_n:\prod_{i=1}^n\pi(i)>0 \}.\]
For some $ \pi \in \cox{D}_n $, it holds that
\[\on{Des}(\pi)=\{0\le i < n: \pi(i)> \pi(i+1) \},\] 
where $ \pi(0)=-\pi(2) $ \cite[Proposition 8.2.2]{BB:CombinatoricsCoxetergroups}. We write for $ \pi \in \cox{D}_n $
\begin{align} \label{eq:descentsTypD}
\on{des}(\pi)=|\on{Des}(\pi)|=\mathbbm{1}_{\{-\pi(2)>\pi(1)\}}+\sum_{i=1}^{n-1}\mathbbm{1}_{\{\pi(i)>\pi(i+1)\}}.
\end{align}

We can reuse the model from the proof of Theorem \ref{t:CLT_B} to generate $ T_{\cox{D}_n} $, with a slight modification: One sign-generating random variable is set to be the product of all the others. Therefore, the number of negative signs is always even. Of course it is not possible to directly apply the method of interaction graphs, as the local dependency structure is destroyed by one random variable being dependent of all the others. This problem is solved via an application of Slutsky's Theorem.

\begin{thm}\label{t:CLT_D}
	Let $ W_n  $ be a sequence of growing rank of Coxeter groups of type $  \cox{D}_n $. Then,
	$ T_{\cox{D}_n} $
	satisfies the central limit theorem, if $ n $ tends to infinity.
\end{thm}
\begin{proof}
	Let $ \mathcal{X}:=[0,1]^2 \times \{-1,1\} $ and $ X_1,X_2,\ldots,X_{n-1} $ be independent and identically distributed of the form $ (U_i,V_i,B_i) $ with $ (U_i,V_i)\sim $Unif$ \left([0,1]^2\right) $ and $ B_i\sim $Ber$ (\frac{1}{2}) $ on $ \{-1,1\} $. Furthermore, set $ X_n=~(U_n,V_n,\prod_{i=1}^{n-1}B_i) $ with $ (U_n,V_n)\sim $Unif$ \left([0,1]^2\right) $ and $ B_n=\prod_{i=1}^{n-1}B_i $. The product of independent Ber$ (\frac{1}{2}) $-distributed random variables on $ \{-1,1\} $ is again Ber$ (\frac{1}{2}) $-distributed on $ \{-1,1\} $. Let $ X:=(X_1,\ldots,X_n) $ and let the x-rank and the y-rank of $ X $ be defined as in the proof of Theorem~\ref{t:CLT_B}.
	$ X_{(1)},\ldots,X_{(n)} $ denote the $ X_i $ ordered in respect to their x-ranks and $ X^{(1)},\ldots,X^{(n)} $ in respect to their y-ranks. Then, as in (\ref{eq:descentsTypD}), if $ \tilde{\pi} \in \cox{D}_n $ and $ \tilde{\pi}^{-1}=\tilde{\sigma} $, with $ \tilde{\pi}(0)=-\tilde{\pi}(2) $ we obtain
	\begin{align*}
	T_{\cox{D}_n}&=\sum_{i=0}^{n-1}\mathbbm{1}_{\{\tilde{\pi}(i)> \tilde{\pi}(i+1)\}}+\sum_{i=0}^{n-1}\mathbbm{1}_{\{\tilde{\sigma}(i)> \tilde{\sigma}(i+1)\}}\\
	&=\mathbbm{1}_{\{-B_{(2)}V_{(2)}>B_{(1)}V_{(1)}\}}+\sum_{i=1}^{n-1}\mathbbm{1}_{\{B_{(i)}V_{(i)}>B_{(i+1)}V_{(i+1)}\}}\\
	&\qquad \qquad  \qquad \qquad \qquad \qquad +\mathbbm{1}_{\{-B^{(2)}U^{(2)}>B^{(1)}U^{(1)}\}}+\sum_{i=1}^{n-1}\mathbbm{1}_{\{B^{(i)}U^{(i)}>B^{(i+1)}U^{(i+1)}\}}.
	\end{align*}
	Now, remove all the indicator functions from $ T_{\cox{D}_n} $ where $ B_{(i)},B^{(i)},B_{(i+1)} $ or $ B^{(i+1)} $ equal $ B_n $ and add indicator functions, so that the resulting random variable is distributed as $ T_{\cox{B}_{n-1}} $.
	Then, as $ \mathbb{E}(T_{\cox{D}_n})=n $ and $ \mathbb{E}(T_{\cox{B}_{n-1}})=n-1 $ (see for example in  \cite{KahleStump2018}),
	\begin{align}
	\frac{T_{\cox{D}_n}-\mathbb{E}(T_{\cox{D}_n})}{\sqrt{\mathbb{V}(T_{\cox{D}_n})}}&=\frac{T_{\cox{B}_{n-1}}+Y_n-n}{\sqrt{\mathbb{V}(T_{\cox{D}_n})}},\nonumber \\
	\intertext{where $ Y_n=T_{\cox{D}_n}-T_{\cox{B}_{n-1}} $ is a random variable with $ |Y_n|\le c $ for some positive constant $ c $ and all $ n $, so}
	\frac{T_{\cox{D}_n}-\mathbb{E}(T_{\cox{D}_n})}{\sqrt{\mathbb{V}(T_{\cox{D}_n})}}&=\frac{{\sqrt{\mathbb{V}(T_{\cox{B}_{n-1}})}}}{{\sqrt{\mathbb{V}(T_{\cox{D}_n})}}}\frac{T_{\cox{B}_{n-1}}-(n-1)}{\sqrt{\mathbb{V}(T_{\cox{B}_{n-1}})}}+\frac{Y_n-1}{{\sqrt{\mathbb{V}(T_{\cox{D}_n})}}}.\label{eq:CLT-D}
	\end{align}
	We know from Theorem \ref{t:CLT_B} that $ \frac{T_{\cox{B}_{n-1}}-(n-1)}{\sqrt{\mathbb{V}(T_{\cox{B}_{n-1}})}} $ converges in distribution to a standard normal distribution. $ Y_n $ is bounded, as it is a finite sum of indicator functions. Therefore, $ \lim\limits_{n\rightarrow \infty} \frac{Y_n-1}{{\sqrt{\mathbb{V}(T_{\cox{D}_n})}}}=0 $ almost surely and $ \lim\limits_{n \rightarrow \infty}\frac{{\sqrt{\mathbb{V}(T_{\cox{B}_{n-1}})}}}{{\sqrt{\mathbb{V}(T_{\cox{D}_n})}}}=~1 $ (compare \cite[Corollary 5.2]{KahleStump2018}). Therefore, $ T_{\cox{D}_n} $ satisfies the central limit theorem (see Slutsky's theorem, for example in \cite[Theorem 2.3.3]{Lehmann1998}).
	
\end{proof}

\section{Generalization to a Class of Statistics with Local Degree \texorpdfstring{$  k $}{k}}\label{s:generalization}

As in \cite{chatterjee2017}, it is possible to generalize the proof of Theorem \ref{t:CLT_B} to a wider class of statistics of local degree $ k $. These statistics are of the form
\[F_1(\pi)+F_2(\pi^{-1}),\]
where the local components' absolute value is bounded by $ 1 $. If $ \pi $ is a signed permutation, a bound for the Wasserstein distance between the normalized statistic and the standard normal distribution follows. Therefore the central limit theorem for these statistics holds, if the variance of the statistics is of order $ n^{\frac{1}{2}+\varepsilon} $ for an $ \varepsilon>0 $. The Theorem is implied from a generalization of the proof of Theorem \ref{t:CLT_B} by constructing the symmetric interaction rule in the right way. 

\begin{thm}\label{t:CLT_arbitrary}
Let $ W_n  $ be a sequence of growing rank of Coxeter groups of type $  \cox{B} $ and let $ F_1,F_2 $ be statistics of local degree $ k $, with the absolute value of their local components bounded by $ 1 $. The statistic $ F_1(\pi)+F_2(\pi^{-1}) $ gives rise to a random variable $ F $. 
The Wasserstein distance between $ F $, normalized by its mean and variance, and the standard normal distribution satisfies
\[\delta_{F}\le C(k) \left( \frac{ \sqrt{n}}{s^2}+\frac{n}{s^3}\right)\]
for $ s^2:=\mathbb{V}(F_1(\pi)+F_2(\pi^{-1})) $ and some constant $ C(k) $.

\end{thm}
\begin{proof}
	If the statistics $ F_1 $ and $ F_2 $ are of local degree $ k $ and their local components' absolute value is bounded by $ 1 $, let $ \{i,j\} $ be an edge in $ G(x) $ if and only if the x-ranks or the y-ranks differ by at most $ k-1 $. For the extension $ G'(x) $, we say that $ \{i,j\} $ is an edge if and only if the ranks differ by at most $ k+3 $. Then, Theorem \ref{t:chatterjee} applies, and the Wasserstein distance is bounded:
	\[\delta_F \le  C(k) \left( \frac{ \sqrt{n}}{s^2}+\frac{n}{s^3}\right).\]
	Here, $ C(k) $ is a large enough constant.
\end{proof}
To see that the bound in Theorem \ref{t:CLT_arbitrary} also holds when  $ \pi $ is an element of a Coxeter group of type $ \cox{D}_n $, we use the same technique as in the proof of Theorem \ref{t:CLT_D}.
Hence, we decompose the statistic into a part that is the same statistic depending on a signed permutation on $ \{\pm 1, \ldots,\pm (n-1)\} $ and a finitely bounded random variable.

\begin{thm}\label{t:CLT_arbitrary-D}
Let $ W_n $ be a sequence of growing rank of Coxeter groups of type $  \cox{D} $ and let $ F_1,F_2 $ be statistics of local degree $ k $, with the absolute value of their local components bounded by $ 1 $. The statistic $ F_1(\pi)+F_2(\pi^{-1}) $ gives rise to a random variable $ F $. Then, if we assume that $ \mathbb{V}(F)\rightarrow \infty $, the Wasserstein distance between $ F $,
normalized by its mean and variance, and the standard normal distribution satisfies
\[\delta_{F}\le C(k) \left( \frac{ \sqrt{n-1}}{s^2}+\frac{n-1}{s^3}\right)+o(1)\]
for $ s^2:=\mathbb{V}(F_1(\pi)+F_2(\pi^{-1})) $ and some constant $ C(k) $.
\end{thm}
\begin{proof}
Let $ F=F_1(\pi_1)+F_2(\pi_1^{-1})=f(X) $ where $ \pi_1 $ is a uniformly chosen element of the Coxeter group of type $ \cox{D}_n $. Let $ X=(X_1,\ldots,X_n) $ be generated as in the proof of Theorem \ref{t:CLT_D}, so $ X_i = (U_i,V_i,B_i) $ with $ (U_i,V_i)\sim $Unif$ \left([0,1]^2\right) $. $ B_i $ is an independent random sign for $ 1\le i \le n-1 $ and $ B_n= \prod_{i=1}^{n-1}B_i $. Then, $ F' $ is the statistic where we remove all local components that depend on $ B_n $. Subsequently we add local components, so that the resulting statistic is $F'=F_1(\pi_2)+F_2(\pi_2^{-1})$, where $ \pi_2 $ is a random signed permutation on $ \{\pm 1,\ldots \pm (n-1)\} $ generated by $ (X_1,\ldots,X_{n-1}) $. Then, as the local degree is~$ k $, $ F-F'=O(1) $ and therefore $ \mathbb{E}(F-F')=O(1) $ and $ \mathbb{V}(F-F')=O(1) $, which implies that $ \mathbb{V}(F')=\mathbb{V}(F)+O(1) $.
Now, see that \cref{eq:CLT-D} from the proof of \cref{t:CLT_D} generalizes to
\begin{align*}
\frac{F-\mathbb{E}(F)}{\sqrt{\mathbb{V}(F)}}&=\frac{{\sqrt{\mathbb{V}(F')}}}{{\sqrt{\mathbb{V}(F)}}}\frac{F'-\mathbb{E}(F')}{\sqrt{\mathbb{V}(F')}}+\frac{F-F'-\mathbb{E}(F-F')}{{\sqrt{\mathbb{V}(F)}}},
\end{align*}
which immediately shows that the Wasserstein distance between $ F $ and $ F' $ tends to zero, because $ \lim\limits_{n \rightarrow \infty} \frac{\mathbb{V}(F')}{\mathbb{V}(F)}=~1 $ and $ \lim\limits_{n \rightarrow \infty}\frac{F-F'-\mathbb{E}(F-F')}{{\sqrt{\mathbb{V}(F)}}}=0 $. Therefore it holds that $ \delta_F\le \delta_{F'}+o(1) $ and the theorem follows.
\end{proof}

\section{The Statistic \texorpdfstring{$ (\on{des}(\pi),\on{des}(\pi^{-1})) $}{(des(pi),des(sigma)}} \label{s:cramerwoldirreducibletypes}
This section derives a two-dimensional central limit theorem for the vector statistic defined as $ (\on{des}(\pi),\on{des}(\pi^{-1})) $ for $ \pi $ being either an element of a Coxeter group of type $ \cox{B}_n $ or $ \cox{D}_n $. This is achieved with the Cram\'er--Wold device and a slight modification of the proofs of Theorems \ref{t:CLT_B} and \ref{t:CLT_D}. The Cram\'er--Wold device shows the equivalence of the convergence in distribution between a random vector and every linear combination of its elements. It is also known as the Theorem of Cram\'er--Wold (see for example in \cite[Theorem~29.4]{billingsley1995probability}). 
\begin{thm}[Cram\'er--Wold]\label{t:cramer-wold}
	Let $ \bar{X}_n=(X_{n1},\ldots,X_{nk}) $ and $ \bar{X}=(X_{1},\ldots,X_{k}) $ be random vectors of dimension $ k $. Then, $ \bar{X}_n \stackrel{D}{\rightarrow}\bar{X} $, if and only if 
	\[\sum_{i=1}^{k}t_i X_{ni}\stackrel{D}{\rightarrow}\sum_{i=1}^{k}t_i X_{i}\]
	for each $ t=(t_1,\ldots,t_k)\in \mathbb{R}^{k} $ and for $ n \rightarrow \infty $.
\end{thm}
We use the short-hand notation $ (D_n ,D'_n) $ for the random variable that rises from $ (\on{des}(\pi),\on{des}(\pi^{-1})) $.
With Theorem \ref{t:cramer-wold}, we can show the convergence of $ (D_n ,D'_n) $ by studying linear combinations of the form $ t_1 D_n+t_2 D'_n $. It is sufficient to only check linear combinations with $ t\in S^{1} $, since the investigated statistic is normalized by the square root of the variance $ \mathbb{V}(t_1 D_n+t_2 D'_n) $. This leads to the following theorem:
\begin{thm}\label{t:CLT_twodimensional}
	Let $ W_n $ be a sequence of Coxeter groups of growing rank of either type $ \cox{B}_n $ or $ \cox{D}_n $. Then, the statistic $ (D_n,D'_n) $ satisfies a two-dimensional central limit theorem of the form
	\[ \Sigma_n^{-\frac{1}{2}} \begin{pmatrix}
	D_n-\mathbb{E}(D_n)\\
	D'_n-\mathbb{E}(D'_n)\\
	\end{pmatrix}\stackrel{D}{\rightarrow}N_2(0,I)\]
	for $ n\rightarrow \infty $, where $ I $ denotes the two-dimensional identity matrix and $ \Sigma_n $ is the covariance matrix of $ (D_n,D'_n) $.
\end{thm}
\begin{proof}
	Via the Theorem of Cram\'er--Wold, we can study the convergence of $ (D_n,D'_n) $ by studying $ t_1 D_n+t_2 D'_n $ for $ t^T=(t_1,t_2)\in S^{1} $. We derive a convergence
	\begin{align}
	t^{T} \frac{1}{\sqrt{\mathbb{V}(D_n)}} \begin{pmatrix}
	D_n-\mathbb{E}(D_n)\\
	D'_n-\mathbb{E}(D'_n)\\
	\end{pmatrix}\stackrel{D}{\rightarrow} N(0,1) \label{eq:2x2-proof_1}
	\end{align}
	to show the Theorem via an application of Slutsky's Theorem. (\ref{eq:2x2-proof_1}) is equivalent to
	\begin{align}
	\frac{1}{\sqrt{\mathbb{V}(D_n)}}(t_1  D_n+t_2 D'_n-(t_1+t_2)\mathbb{E}(D_n)) \stackrel{D}{\rightarrow} N(0,1),\label{eq:2x2-proof_2}
	\end{align}
	as $ \mathbb{E}(D_n) =\mathbb{E}(D'_n)$. Now, since $ t\in S^{1} $, the proofs of the Theorems \ref{t:CLT_B} and \ref{t:CLT_D} apply, which means that 
	\[ \frac{t_1  D_n+t_2 D'_n-(t_1+t_2)\mathbb{E}(D_n)}{\sqrt{\mathbb{V}(t_1  D_n+t_2 D'_n)}}\stackrel{D}{\rightarrow} N(0,1). \]
	This convergence is also a consequence of Theorem \ref{t:CLT_arbitrary} or Theorem \ref{t:CLT_arbitrary-D}, as the local components of $ t_1  D_n+t_2 D'_n $ are still bound by $ 1 $ and the local dependency structure is not changed by multiplying the sum of indicator functions that model $ D_n $ and $ D'_n $ with constants. Furthermore, the variance $ \mathbb{V}(t_1  D_n+t_2 D'_n) $ is of order $ n $ and therefore, the Wasserstein distance to the standard normal distribution is bound by a vanishing function in $ n $.
	Now, by Slutsky's Theorem, (\ref{eq:2x2-proof_2}) and therefore (\ref{eq:2x2-proof_1}) is satisfied as 
	\[\frac{\mathbb{V}(t_1 D_n+t_2 D'_n)}{\mathbb{V}(D_n)}\stackrel{a.s}{\rightarrow}1.\]
	This results from the fact that $ \mathbb{V}(D_n)=\mathbb{V}(D'_n)) $ and $ \text{Cov}(D_n,D'_n)=O(1) $ (see \cite{KahleStump2018}) and that $ t_1^2+t_2^2=1 $. Because of the convergence in (\ref{eq:2x2-proof_1}), the theorem follows via another application of Slutsky's Theorem, as
	\begin{align*}
	\frac{1}{\mathbb{V}(D_n)}\Sigma_n &=\frac{1}{\mathbb{V}(D_n)}\begin{pmatrix}
	\mathbb{V}(D_n)&&\text{Cov}(D_n,D'_n)\\
	\text{Cov}(D_n,D'_n)&&\mathbb{V}(D'_n)\\
	\end{pmatrix}\\
	&\stackrel{a.s.}{\rightarrow}I,
	\end{align*}
	since $ \text{Cov}(D_n,D'_n)=O(1) $ and $ \mathbb{V}(D_n)=\mathbb{V}(D'_n) $.
\end{proof}
\begin{remark}
	Theorem \ref{t:CLT_twodimensional} can be generalized to certain statistics $ (F_1(\pi),F_2(\pi^{-1})) $, if $ F_1 $ and $ F_2 $ meet the constraints of Theorem \ref{t:CLT_arbitrary} or Theorem \ref{t:CLT_arbitrary-D}, $ \mathbb{V}(F_1(\pi))=\mathbb{V}(F_2(\pi^{-1})) $ holds and $ \mathbb{V}(F_1(\pi)) $ is big enough so that the constraint to the Wasserstein distance in Theorem \ref{t:CLT_arbitrary} or Theorem \ref{t:CLT_arbitrary-D} converges to zero for $ n $ going to infinity.
\end{remark}

\section{Further Investigation}

This paper showed the central limit behaviour for $ D(\pi)+D(\pi^{-1}) $, where $ \pi $ is an element of either a Coxeter group of type $ \cox{B}_n $ or of type $ \cox{D}_n $. 
A natural direction for further investigation are arbitrary series of product groups of Coxeter groups of type $ \cox{A}_n,\cox{B}_n $ and $ \cox{D}_n $ and under which constraints the asymptotic normality of $ D(\pi)+D(\pi^{-1}) $ is preserved (see Problem 6.10 in \cite{KahleStump2018}). 
By November 2019, using the results of this paper, this was done by Brück and R\"ottger \cite{BR2019} and Féray \cite{Feray2019}.

\bibliographystyle{plain}
\bibliography{bibliography}

\end{document}